\documentclass[a4paper]{article}

\usepackage{
amsmath,
amsthm,
amscd,
amssymb,
}
\usepackage{mathtools}
\usepackage{comment}
\usepackage{tikz}
\usetikzlibrary{positioning,arrows,calc}
\usepackage{xspace}

\usepackage{graphicx}

\usepackage{url}

\theoremstyle{plain}
\newtheorem{lemma}{Lemma}
\newtheorem{definition}{Definition}

\newtheorem{theorem}{Theorem}

\newtheorem{question}{Question}

\newtheoremstyle{derp}
{3pt}
{3pt}
{}
{}
{\upshape}
{:}
{.5em}
{}
\theoremstyle{derp}
\newtheorem{example}{Example}

\newcommand{\R}{\mathbb{R}}

\newcommand{\Z}{\mathbb{Z}}

\newcommand{\N}{\mathbb{N}}

\newcommand{\supp}{\mathrm{supp}}

\title{Trees are nilrigid}

\author{
Ville Salo \\
vosalo@utu.fi
}

\begin{document}
\maketitle

\begin{abstract}
We study cellular automata on the unoriented $k$-regular tree $T_k$, i.e. continuous maps acting on vertex-labelings of $T_k$ which commute with all automorphisms of the tree. We prove that every CA that is asymptotically nilpotent, meaning every configuration converges to the same constant configuration, is nilpotent, meaning each configuration is mapped to that configuration after finite time. We call group actions nilrigid when their cellular automata have this property, following Salo and T\"orm\"a. In this terminology, the full action of the automorphism group of the $k$-regular tree is nilrigid. We do not know whether there is a nilrigid automorphism group action on $T_k$ that is simply transitive on vertices.
\end{abstract}

\section{Definitions}

\begin{definition}
Let $T$ be a set, $G \curvearrowright T$ a group action. A \emph{cellular automaton (CA)} on $(G,T)$ is a continuous function $f : A^T \to A^T$, where $A$ is a finite set called the \emph{alphabet} and $A^T$ has the product topology, such that $f \circ g = g \circ f$ for all $g \in G$, where the action $G \curvearrowright A^T$ is defined by $g x_h = x_{g^{-1} h}$. 
\end{definition}

For $k \geq 2$, write $T_k$ for the infinite (rootless) leafless undirected $k$-regular tree, and $G_k$ for its automorphism group (group of bijections on vertices that preserve the edges), a discrete\footnote{This group admits a natural locally compact topology, but this plays no role here.} uncountable group.

One set-theoretic implementation of $T_k$ is obtained by taking the vertices $V$ to be finite words over the alphabet $\{0,1,...,k-1\}$, including the empty word, such that $k-1$ can only occur as the first symbol, i.e.\ the regular language $V = \{\epsilon\} \cup \{0,1,...,k-1\} \{0,1,...,k-2\}^*$, and undirected edges $E$ the sets $\{w,wa\}$ where $w, wa \in V$ and $a \in \{0,1,...,k-1\}$. Another implementation is the undirected Cayley graph of the free product $\Z_2 * \Z_2 * \cdots \Z_2$ of $k$ copies of the two-element group $\Z_2$, with generators those of the $\Z_2$.

Slightly abusing notation, we write $A^{T_k}$ for the set of all functions $\alpha : V \to A$ where $V$ is the set of vertices of $T_k$. By similar abuse of notation, cellular automata on the $k$-ary tree are cellular automata $f : A^{T_k} \to A^{T_k}$ on the action $G_k \curvearrowright T_k$. 

\begin{definition}
Let $X$ be a topological space and $f : X \to X$ a function. We say $f$ is \emph{asymptotically nilpotent (towards $0 \in X$)} if $f^n(x) \rightarrow 0$ for all $x \in X$. We say $f$ is \emph{uniformly asymptotically nilpotent} if it is asymptotically nilpotent, and the convergence is uniform over $X$. We say $f$ is \emph{nilpotent (towards $0 \in X$)} if there exists $n \in \N$ such that $f^n(x) = 0$ for all $n \in \N$. A point $x \in X$ is \emph{mortal (towards $0 \in X$)} if there exists $n \in \N$ such that $f^n(x) = 0$.
\end{definition}

\begin{definition}
We say a group action $G \curvearrowright T$ is \emph{nilrigid} if for every finite alphabet $A$, every asymptotically nilpotent CA $f : A^T \to A^T$ is uniformly asymptotically nilpotent. A group $G$ is \emph{nilrigid} if the regular action $G \curvearrowright G$ defined by $g \cdot g' = gg'$ is nilrigid.
\end{definition}

For $a \in A$ and $T$ a set, write $a^T$ for the unique element of $\{a\}^T$. When studying cellular automata on transitive group actions, any type of nilpotency of $f : A^T \to A^T$ defines the point $0 \in X$ in a shift-invariant way, thus necessarily $0 = 0^T$ for some $0 \in A$. It is easy to see that if $G \curvearrowright T$ is transitive then a CA $f : A^T \to A^T$ is nilpotent if and only if it is uniformly asymptotically nilpotent.

For $x \in A^\Z$ or $x \in A^\N$, define $\sigma(x)$ by the formula $\sigma(x)_i = x_{i+1}$. If a symbol $0 \in A$ is clear from context, for any set $T$ and $x,y \in A^T$ such that $\forall t \in T: 0 \in \{x_t, y_t\}$, write $x + y$ for the unique configuration satisfying $\forall t \in T: \{0, (x + y)_t\} = \{x_t, y_t\}$. (This is the cellwise sum of $x$ and $y$ with respect to any abelian group structure on $A$.)

\section{Lemmas}

We prove some basic lemmas about asymptotic nilpotency.

\begin{lemma}
\label{lem:SyndeticallyClose}
Let $X$ be a compact metrizable space and $f : X \to X$ a function. If $f$ is asymptotically nilpotent towards $0$ then
\[ \forall \epsilon > 0: \exists N \in \N: \forall x \in X: \exists n \leq N: d(f^n(x), 0) < \epsilon. \]
\end{lemma}

\begin{proof}
Suppose not, i.e. there exists $\epsilon > 0$ such that
\[ \forall N \in \N: \exists x_N \in X: \forall n \leq N: d(f^n(x), 0) \geq \epsilon. \]
Let $x = \lim_i x_{N_i}$ for some sequence $N_1 < N_2 < N_3 < \cdots$. Then $d(f^n(x), 0) \geq \epsilon$ for all $n \in \N$.
\end{proof}

Associate to a function $f : A^T \to A^T$ a bipartite graph $G_f$ where the left and right vertex sets are copies of $T$, and we have an edge $(t,t')$ if there exist two configurations $x,y \in A^T$ such that $x_u = y_u$ for all $u \neq t$ and $f(x)_{t'} \neq f(y)_{t'}$. Say $f$ \emph{has nice local rules} if the graph $G_f$ is locally finite. This means that $f(x)_t$ only depends on finitely many cells of $x$, and only finitely many cells of $f(x)$ depend on a particular coordinate $x_t$. Observe that having nice local rules implies continuity. 

\begin{lemma}
Suppose $f : A^T \to A^T$ is continuous and commutes with a transitive group action $G \curvearrowright T$, and the stabilizers of elements $t \in T$ in the action of $G$ have finite orbits. Then $f$ nice local rules.
\end{lemma}

\begin{proof}
Since $f$ is continuous, every $t \in T$ has a unique \emph{minimal neighborhood}, a finite subset $N(t) \subset T$ such that $f(x)_t$ only depends on $x|_{N(t)}$, and $N(t)$ is minimal with this property. Since $f$ commutes with $G$, by the uniqueness of the minimal neighborhood we have $gt = t' \implies gN(t) = N(t')$. It is clear that the graph described in the definition of nice local rules simply records the edges from $N(t)$ to $t$ for all $t$. Since $f$ is continuous, the only way for it not to have nice local rules, is that there exists some $t \in T$ such that $t \in N(t')$ for infinitely many $t'$.

Since $G$ acts transitively, we can take any such $t' \in T$ and find infinitely many elements of $G$, say $g_1, g_2, ...$, such that $t \in N(g_i t') = g_i N(t')$, such that all $g_i t'$ are distinct. By the pigeonhole principle, there is some $t'' \in N(t')$ such that $g_i t'' = t$ for infinitely many $g_i$. Now observe that for all $i, j$, the translation $g_i g_j^{-1} \in G$ fixes $t$, but obviously $g_1 t'$ has infinite orbit under the action of the group $\langle g_i g_j^{-1} \rangle$, meaning the stabilizer of $t$ does not have finite orbits.
\end{proof}

Transitivity of the action does not suffice in the previous statement, as shown by the following example. A similar example was constructed by T\"orm\"a (private communication).

\begin{example}
Consider the set $X_0 \times \Z$ where $X_0 \subset \Z^\N$ is the set of sequences with finite support, seen as an infinite tree so the parent of $(x,n)$ is $(\sigma(x),n+1)$ and the children are $(kx,n-1)$, $k \in \Z$. The free group $\langle a, b \rangle$ acts by shifting the main branch, by $a \cdot (x,n) = (x, n+1)$, and permuting (by $\Z$-translation) the branch at a fixed level, by $b \cdot (x,n) = (y,n)$ where $y_{-n} = x_{-n}+1$ if $n \geq 0$ and $y_i = x_i$ for all $i \in \N \setminus \{-n\}$. This action is transitive, and the map $f(\gamma)_{(x,n)} = \gamma_{(\sigma(x),n+1)}$, which copies the value from the parent to the children, commutes with the action of $\langle a,b \rangle$, is continuous, yet does not have nice local rules.
\end{example}

We now prove Lemma~\ref{lem:FiniteMortalNilpotent} which says that asymptotic nilpotency implies nilpotency if there are lots of finite mortal configurations. A more subtle version of this argument can be found in \cite{GuRi10}. The \emph{support (w.r.t.\ $0 \in A$)} of a configuration $x \in A^T$ is $\supp(x) = \{t \in T \;|\; x_t \neq 0\}$. We say $x$ is \emph{finite} if its support is.

\begin{lemma}
\label{lem:FiniteMortalNilpotent}
Let $T$ be a set and suppose $f : A^T \to A^T$ have nice local rules. If $f$ is asymptotically nilpotent towards $0^T$ and finite mortal configurations are dense in $X$, then $f$ is uniformly asymptotically nilpotent.
\end{lemma}

\begin{proof}
Since $f$ is asymptotically nilpotent, by the Baire category theorem, for some $m$ there exists an open set $U$ such that every configuration $x \in U$ satisfies $f^{k}(x)_{t} = 0$ for all $k \geq m$ (see \cite{GuRi10}.) Let $y \in U$ be a finite mortal configuration, and suppose $f^{m'}(y) = 0^T$.

Suppose $f$ is not uniformly asymptotically nilpotent. Then for arbitrarily large $k \in \N$, there exists $x \in X$ such that $f^k(x)_t \neq 0$ (since $f$ is asymptotically nilpotent this is true for some $t' \in T$, and we can conjugate by a translation to have it true for $t$). By the previous lemma, due to asymptotic nilpotency for any finite set $B \subset T$ there exists $N \in \N$ such that for any $x \in A^T$ there exists $n \leq N$ such that $f^n(x)|_B = 0^B$. Combining these facts, for any finite $B \subset T$ there exist arbitrarily large $k$ and $x \in A^T$ such that $x|_B = 0^B$ and $f^k(x)_t \neq 0$.

Since $f$ has nice local rules, if $B$ is large enough, for the configuration $y$ chosen as in the first paragraph and $x$ chosen as in the previous paragraph, we have $f^n(x + y) = f^n(x) + f^n(y)$ for all $n \in \N$. Namely, since $y$ is mortal, if this holds for all $n < m'$, it holds for all $n \in \N$. Now simply pick $B$ large enough: $f$ has nice local rules, clearly there is no interaction between $x$ and $y$ in the first $m'$ steps. Now, observe that if $x$ is picked as in the previous paragraph so that $f^{k}(x)_{t} \neq 0$ for some $k \geq m$, then if $B$ is large enough we have $x + y \in U$ and $f^k(x + y)_t = f^k(x)_t + f^k(y)_t = f^k(x)_t \neq 0$, a contradiction.
\end{proof}

We say the \emph{support of $x \in A^T$ stays uniformly bounded} if there exists a finite set $S \subset T$ such that $\supp(f^n(x)) \subset S$ for all $n \in \N$. The following lemma is an obvious combination of Lemma~\ref{lem:SyndeticallyClose} and Lemma~\ref{lem:FiniteMortalNilpotent}.

\begin{lemma}
Let $T$ be a set and suppose $f : A^T \to A^T$ have nice local rules. If $f$ is asymptotically nilpotent towards $0^T$ and every the support of every finite configuration stays uniformly bounded in the iteration of $f$, $f$ is uniformly asymptotically nilpotent.
\end{lemma}

We also recall the following result of Guillon and Richard.

\begin{theorem}[\cite{GuRi10}]
\label{thm:ZNilrigid}
The group $\Z$ is nilrigid.
\end{theorem}

\section{The theorem}

\begin{theorem}
The action $G_k \curvearrowright T_k$ of the automorphism group of the $k$-regular tree $T_k$ is nilrigid.
\end{theorem}

\begin{proof}
Write $G = G_k$, $T = T_k$. Suppose $f$ is a CA on $A^{T}$ and $0 \in A$, and $f^n(x) \rightarrow 0^{T}$ for all $x \in A^{T}$. We need to show that for some $n \in \N$, we have $f^n(x) = 0^{T}$ for all $x \in A^{T}$. Write $d : T \times T \to \R$ for the path metric in $T_k$.

Pick a bi-infinite injective path $p : \Z \to T$, i.e.\ for all $n \in \Z$, $\{p(n), p(n+1)\}$ is an edge in $T$ and $p(n) \neq p(n+2)$. We stratify $T$ according to distance from ``the end $p(-\infty)$'', normalized so $p(0)$ is at level $0$. More precisely, we define the Busemann function $b : T \to \Z$ by 
\[ b(t) = \lim_{n \rightarrow \infty} d(p(-n), t) - n. \]

Define $Y \subset A^T$ by $y \in Y \iff \forall t, t' \in T: b(t) = b(t') \implies y_t = y_{t'}$. Thinking of the tree as growing upward from the end $p(-\infty)$, this means the colors of vertices on the same height have the same color. Define a bijection $\phi : A^\Z \to Y$ by $\phi(x)_t = x_{b(t)}$. We have $f(Y) \subset Y$: if $b(t) = b(t')$ pick a large enough $n$ so $t$ and $t'$ are both further from $p(-\infty)$ than $p(-n)$, i.e.\ the limits defining $b$ has been reached. Pick an automorphism $g : T \to T$ that fixes $p(-n')$ for all $n \geq n'$ and maps $t$ to $t'$. Observe that $g$ fixes every configuration $y \in Y$, so for all $y \in Y$ we have
\[ f(y)_t = f(y)_{g^{-1} t'} = gf(y)_{t'} = fg(y)_{t'} = f(y)_{t'}. \]
This is simply the observation that an endomorphism cannot make the stabilizer subgroup of a configuration smaller.

Note that $\phi^{-1}(y)_i = y_{p(i)}$ for $y \in Y$. Observe that $\phi \circ \sigma \circ \phi^{-1} : Y \to Y$ is induced by any automorphism $g$ of the tree mapping $p(n)$ to $p(n-1)$ for all $n \in \Z$. Similarly, $x \mapsto \phi^{-1}f\phi(x)$ defines a function $\bar f : A^\Z \to A^\Z$, and because $f$ commutes with $g$, $\bar f$ commutes with $\sigma = \phi^{-1} \circ g \circ \phi$. Since $\bar f$ is obviously continuous, it is a cellular automaton.

Since $f$ is asymptotically nilpotent, it is asymptotically nilpotent on $Y$, and it is easy to see that $\bar f$ is then asymptotically nilpotent as well. It follows from Theorem~\ref{thm:ZNilrigid} that for some $n \in \N$, $\bar f^n(x) = 0^T$ for all $x \in X$. Thus $f^n(y) = 0^T$ for all $y \in Y$.

Clearly $f^n$ is (uniformly) asymptotically nilpotent if and only if $f$ is, so we can w.l.o.g. replace $f$ by $f^n$. So suppose $f(y) = 0^T$ for all $y \in Y$. It follows that for any $t \in T$, for some $\epsilon > 0$, whenever we have $d(x, y) < \epsilon$ for some $y \in Y$, then $f(x)_t = 0$.

By Lemma~\ref{lem:SyndeticallyClose}, it is enough to show that the support of any finite configuration $x$ stays uniformly bounded. It is easy to see that if $x$ and $f(x)$ both have support of size one, then their supports are equal. Thus it is enough to show that finite configurations wth finite support of size at least $2$ stay uniformly bounded. Let thus $x$ be finite with support size at least $2$ and $S' \subset T$ its finite support. Let $S \subset T$ be the ``convex hull'' of $S'$, i.e.\ the smallest set containing $S'$ and the shortest path between any pair $s, s' \in S'$.

Consider $S$ as an induced subgraph of $T$, and let $F \subset S$ be its set of leaves. For each $u \in F$, let $B_u$ be the connected component starting from $u$ outside $S \setminus F$, i.e.\
\[ B_u = \{t \in (T \setminus S) \cup \{u\} \;|\; \forall u' \in F \setminus \{u\}: d(t, u) < d(t, u'). \]
Define a map $c : B_u \to \N$ by $c(u) = 0$ and in general $c(t) = d_T(t, u)$.

Let $Z_u$ be the set of configurations $z$ where the colors in the component $B_u$ only depend on the value of $c(t)$, i.e.
\[ z \in Z_u \iff (\forall t, t' \in B_u: c(t) = c(t') \implies z_g = z_h). \]
Let $Z = \bigcap_{u \in F} Z_u$. We claim that $f(Z) \subset Z$. For this, it is enough to show $f(Z_u) \subset Z_u$ for all $u \in F$. For this, observe that the subgroup of $G_k$ that stabilizes $T \setminus B_u$ acts transitively on $c^{-1}(n)$ for all $n \geq 0$. Thus $f(Z_u) \subset Z_u$ is proved by a similar formula as $f(Y) \subset Y$.

Suppose then $z \in Z$. We claim that there exists $n$ such that whenever $d(t, u) \geq n$, then $f(z)_t = 0$. To see this, it is enough to show that for all $u \in F$ there exists $n$ such that whenever $z \in Z_u$, $t \in B_u$ and $d(t, u) \geq n$, we have $f(z)_t = 0$.

To prove this, consider one such pair $d(t, u) = m \geq n$. Consider any automorphism $g \in G_k$ that maps $u$ to $p(-m)$ and $t$ to $p(0)$, and in general if $t' \in B_u$ and $c(t') = i$, maps $t'$ into $b^{-1}(i - m)$ (so $gB_u = b^{-1}([-m,\infty))$.) We can find a configuration $y \in Y$ which agrees with $g z$ in every cell at distance at most $m$ from $p(0)$: for any $q : (-\infty, -m-1] \to A$ there is a unique configuration $y \in Y$ that satisfies $y|_{gB_u} = gz|_{gB_u}$ and $y_|{p(i)} = q(i)$ for $i < -m$. If $n \leq m$ is large enough, we then have $f(gz)_{p(0)} = f(y)_{p(0)} = 0$, so
\[ 0 = f(gz)_{p(0)} = f(z)_{g^{-1} p(0)} = f(z)_t. \]

Now, we see that $z \in Z$ implies that after an application of $f$, every nonzero symbol is at distance at most $n$ from $S$ along the tree. Now, if the support of $x$ is originally contained in $S$, then obviously $x \in Z$, so for all $n$ the support of $f^n(x)$ stays within distance $n$ from $S$.
\end{proof}

\section{Questions}

The infinite $k$-regular tree is a Cayley graph for $k \geq 2$: for any $k$, it is the Cayley graph of $k$ copies of the two-element group $\Z_2$, and for even $k$, it is the Cayley graph of the free group $F_{k/2}$ with respect to a free generating set of size $k/2$.

\begin{question}
Does $T_k$ admit a simply transitive nilrigid group action by graph automorphisms?
\end{question}

The related question of nilrigidity $F_k$ was asked in \cite{SaTo18}. We also repeat \cite[Question~11.8]{SaTo18}:

\begin{question}
Is every transitive group action nilrigid?
\end{question}

One can also formulate the question for monoid actions (the opposite monoid acts naturally on configurations). We do not know whether all transitive monoid actions are nilrigid.

\bibliographystyle{plain}
\bibliography{../../../bib/bib}{}

\end{document}